\documentclass[12pt]{article}
\usepackage{amsmath,amssymb,amsthm}
\usepackage{amscd}
\usepackage[dvips, dvipdfmx]{graphicx}
\usepackage{color}
\usepackage{fancyhdr}
\usepackage{ifthen}
\usepackage{mathrsfs}
\usepackage{datetime} 

\newtheorem{thm}{Theorem}[section]
\newtheorem{lemma}[thm]{Lemma}

\newtheorem{prop}[thm]{Proposition}
\newtheorem{cor}[thm]{Corollary}

\theoremstyle{definition}
\newtheorem{df}[thm]{Definition}
\newtheorem{ex}[thm]{Example}

\theoremstyle{remark}
\newtheorem{rem}[thm]{Remark}

\newcommand{\va}{{\bf a}}

\newcommand{\vzero}{{\bf 0}}

\newcommand{\Z}{\mathbb Z}
\newcommand{\Q}{\mathbb Q}

\newcommand{\Af}{\mathbb{A}_{\rm f}}

\newcommand{\hZ}{\widehat{\Z}}

\newcommand{\Aut}{{\rm Aut}}
\newcommand{\End}{{\rm End}}

\newcommand{\diag}{{\rm diag}}

\newcommand{\bs}{\backslash}

\newcommand{\Hei}{\mathcal{H}}
\newcommand{\gHei}{\mathfrak{H}}

\newcommand{\Fr}{\mathcal{F}}
\newcommand{\gFr}{\mathfrak{F}}
\newcommand{\hl}{\widehat{L}}
\newcommand{\lp}{L_p}

\newcommand{\halpha}{\widehat{\alpha}}

\newcommand{\trpi}{T_{r,p}^{(i)}}

\newcommand{\gl}{G_L}
\newcommand{\dl}{\Delta_L}
\newcommand{\gaml}{\Gamma_{\! L}}

\newcommand{\glp}{G_{L_p}}
\newcommand{\dlp}{\Delta_{L_p}}
\newcommand{\gamlp}{\Gamma_{\! L_p}}

\newcommand{\hgl}{\widehat{G}_L}
\newcommand{\hdl}{\widehat{\Delta}_L}
\newcommand{\hgaml}{\widehat{\Gamma}_{\! L}}

\newcommand{\tl}{T_{L}}
\newcommand{\tlp}{T_{L_p}}
\newcommand{\htl}{\widehat{T}_L}
\newcommand{\tln}{\tl(n)}
\newcommand{\htln}{\htl(n)}
\newcommand{\tlpk}{\tlp(p^k)}

\newcommand{\dls}{D_L(s)}
\newcommand{\plp}{P_{L_p}}

\newcommand{\etal}{\eta_L}
\newcommand{\etadl}{\eta_{*L}}
\newcommand{\etaul}{\eta_L^*}

\newcommand{\rl}{R_L}
\newcommand{\rlp}{R_{L_p}}
\newcommand{\rlq}{R_{L_q}}
\newcommand{\hrl}{\widehat{R}_L}

\newcommand{\gHp}{G_{\mathcal{H}_p}}
\newcommand{\dHp}{\Delta_{\mathcal{H}_p}}
\newcommand{\gamHp}{\Gamma_{\! \mathcal{H}_p}}

\newcommand{\hdH}{\widehat{\Delta}_\mathcal{H}}

\newcommand{\tHp}{T_{\mathcal{H}_p}}
\newcommand{\tHq}{T_{\mathcal{H}_q}}

\newcommand{\dHs}{D_\mathcal{H}(s)}
\newcommand{\hdHs}{\widehat{D}_\mathcal{H}(s)}
\newcommand{\pHp}{P_{\mathcal{H}_p}}

\newcommand{\etadH}{\eta_{*\mathcal{H}}}
\newcommand{\etauH}{\eta_\mathcal{H}^*}

\newcommand{\rH}{R_\mathcal{H}}
\newcommand{\rHp}{R_{\mathcal{H}_p}}
\newcommand{\hrH}{\widehat{R}_\mathcal{H}}

\newcommand{\gZr}{G_{\Z^r}}
\newcommand{\dZr}{\Delta_{\Z^r}}
\newcommand{\gamZr}{\Gamma_{\! \Z^r}}

\newcommand{\gZrp}{G_{{\Z^r_p}}}
\newcommand{\dZrp}{\Delta_{\Z^r_p}}
\newcommand{\gamZrp}{\Gamma_{\! \Z^r_p}}

\newcommand{\hdZr}{\widehat{\Delta}_{\Z^r}}
\newcommand{\hgamZr}{\widehat{\Gamma}_{\! \Z^r}}

\newcommand{\dZrs}{D_{\Z^r}(s)}
\newcommand{\hdZrs}{\widehat{D}_{\Z^r}(s)}
\newcommand{\pZrp}{P_{\Z^r_p}}

\newcommand{\etaZr}{\eta_{\Z^r}}
\newcommand{\etadZr}{\eta_{*{\Z^r}}}
\newcommand{\etauZr}{\eta_{\Z^r}^*}

\newcommand{\gZtp}{ G_{ {\Z^2_p} } }
\newcommand{\dZtp}{\Delta_{\Z^2_p}}
\newcommand{\gamZtp}{\Gamma_{\! \Z^2_p}}

\newcommand{\rZr}{R_{\Z^r}}
\newcommand{\rZrp}{R_{\Z^r_p}}
\newcommand{\hrZr}{\widehat{R}_{\Z^r}}

\newcommand{\rZtp}{R_{\Z^2_p}}
\newcommand{\hrZt}{\widehat{R}_{\Z^2}}

\newcommand{\hdZt}{\widehat{D}_{\Z^2}}
\newcommand{\pZtp}{P_{\Z^2_p}}
\newcommand{\tZtp}{T_{\Z^2_p}}
\newcommand{\htZt}{\widehat{T}_{\Z^2}}

\newcommand{\hs}{\widehat{s}}
\newcommand{\hphi}{\widehat{\phi}}
\newcommand{\hpsi}{\widehat{\psi}}
\newcommand{\htheta}{\widehat{\theta}}
\newcommand{\htH}{\widehat{T}_{\Hei}}

\newcommand{\hbtheta}{\widehat{\boldsymbol{\theta}}}

\newcommand{\hdzt}{\widehat{\Delta}_{\Z^2}}
\newcommand{\hdls}{\widehat{D}_L(s)}

\newcommand{\mo}{\Delta}

\newcommand{\gr}{\Gamma}
\newcommand{\degr}[1]{\deg \! {#1}}

\newcommand{\zetai}[1]{\zeta^i_{#1}}
\newcommand{\zetah}[1]{\zeta_{#1}^{\wedge}}

\newcommand{\ani}[1]{a_n^i(#1)}
\newcommand{\anh}[1]{a_n^{\wedge}(#1)}

\newcommand{\zetaiLs}{\zetai{L}(s)}
\newcommand{\zetahLs}{\zetah{L}(s)}

\newcommand{\aniL}{\ani{L}}
\newcommand{\anhL}{\anh{L}}

\newcommand{\zetaiHs}{\zetai{\Hei}(s)}
\newcommand{\zetahHs}{\zetah{\Hei}(s)}

\newcommand{\Si}{\mathcal{S}^i}
\newcommand{\Sh}{\mathcal{S}^{\wedge}}

\newcommand{\tSh}{\mathcal{S}'^{\wedge}}

\begin{document}
%\pagewiselinenumbers
\title{Global properties of a Hecke ring associated with the Heisenberg Lie algebra}
%\subtitle{Do you have a subtitle?\\ If so, write it here}

%\titlerunning{A formal power series over a noncommutative Hecke ring}        % if too long for running head

\author{Fumitake Hyodo}

%\authorrunning{Short form of author list} % if too long for running head

%\institute{Fumitake Hyodo \at
% Department of Health Informatics, Faculty of Health and Welfare Services Administration, Kawasaki University of Medical Welfare, Kurashiki, 701-0193, Japan \\
%              Tel.: +81-86-462-1111\\
%              Fax: +81-86-462-1193\\
%              \email{fumitake.hyodo@mw.kawasaki-m.ac.jp}           %  \\
%%             \emph{Present address:} of F. Author  %  if needed
%}

\date{}
%\date{\today}
% The correct dates will be entered by the editor

\maketitle

%%%%%%%%%%%%%%%%%
\begin{abstract}
This study concerns (not necessarily commutative) Hecke rings associated with certain algebras
and describes a formal Dirichlet series with coefficients in the Hecke rings, which can be used to generalize Shimura's series. Considering the case of the Heisenberg Lie algebra, an analog of the identity for Shimura's series derived employing the rationality theorem, presented by Hecke and Tamagawa, is established. Moreover, this analog recovers the explicit formula for the pro-isomorphic zeta function of the Heisenberg Lie algebra shown by Grunewald, Segal and Smith.
\end{abstract}
%\bigskip

\renewcommand{\thefootnote}{\fnsymbol{footnote}}
\footnote[0]{
%\begin{description}
%\item{Mathematics Subject Classification.} Primary 20C08; Secondary 20G25, 20G30, 11F03.
2020 Mathematics Subject Classification. Primary 20C08; Secondary 20G25, 20G30, 11F03, 11M41, 20E07.
%\item{Key words and phrases.} Hecke rings, noncommutative Hecke rings, Dirichlet series, \\Heisenberg Lie algebra.
%\end{description}
}
%20C08 Hecke algebras and their representations
%20G25 Linear algebraic groups over local fields and their integer
%20G30 Linear algebraic groups over global fields and their integers
%11M41 Other Dirichlet series and zeta functions 
%11F03 Modular and automorphic functions
%20E07 Subgroup theorems; subgroup growth

\section{Introduction}\label{section:intro}
This study concerns Hecke rings introduced by Shimura \cite{S1}.
A classical study of Hecke rings is the work by Hecke \cite{H} and Tamagawa \cite{T} on the Hecke rings associated with the general linear groups. They showed that these Hecke rings are commutative polynomial rings. Furthermore, they defined formal power series with coefficients in these Hecke rings, and showed their rationality.
The results of this work are summarized in \cite[Chapter 3]{S2}, where formal Dirichlet series with coefficients in these Hecke rings were further introduced.
Andrianov \cite{A1}, Hina--Sugano \cite{Hi}, Satake \cite{Sa}, and Shimura \cite{S3} studied Hecke rings associated with classical groups, wherein they further developed the work of Hecke \cite{H} and Tamagawa \cite{T}.
In addition, other studies were conducted on the Hecke rings associated with Jacobi and Chevalley groups by Dulinsky \cite{D} and Iwahori-Matsumoto \cite{IM}, respectively. 

As mentioned above, various studies have been carried out on Hecke rings. However, the class of Hecke rings defined by Shimura is vast, and only a small part of it has been studied to date. 

From now on, an algebra implies an abelian group with a bi-additive product (e.g., an associative algebra, a Lie algebra).
Let $L$ be an algebra that is free of finite rank as an abelian group. Our previous work \cite{Hy2} introduced the Hecke rings $\rl$ and $\hrl$ associated with $L$. For the definition, see Section \ref{section:heckerings}. 
In this study, we deal with the formal Dirichlet series $\dls$ and $\hdls$ with coefficients in $\rl$ and $\hrl$, respectively, which are defined in Section \ref{section:seriesalg}.

The first result of this study is to show that the Euler product formula for $\hdls$ holds, and to give a sufficient condition for $\dls$ to have the Euler product expansion (cf. Theorems \ref{thm:eulerprodhdls} and \ref{thm:eulerproddls}).

If $L = \Z^r$ is the free abelian Lie algebra of rank $r$, the Hecke ring $\rZr$ and $\hrZr$ coincide with those treated by Hecke \cite{H} and Tamagawa \cite{T}. Further, the formal Dirichlet series $\dZrs$ and $\hdZrs$ equal those treated in \cite[Chapter 3]{S2}. Thus, it can be said that our study generalizes their study. We discuss them in Section \ref{section:Zr}.

Denote by $\Hei$ the Heisenberg Lie algebra, that is, the free nilpotent Lie algebra of class $2$ on two generators. The second result of this study is the establishment of identities for $\dHs$ and $\hdHs$, which is the primary result of this study.  Let $\hbtheta = (\htheta_p)_p$ be a family of indeterminates indexed by all prime numbers $p$.
The key idea for stating our main theorem involves regarding $\hrH$ as a module over the polynomial ring $\hrZt[\hbtheta]$. The main theorem is as follows:
\begin{thm}[Theorem \ref{thm:globalequality}]\label{thm:global:Heisenbergex:intro}
There exists a formal Dirichlet series $\widehat{I}_2 (\hbtheta;s)$ with coefficients in $\hrZt[\hbtheta]$ satisfying the following identity:
\[\widehat{I}_2 (\hbtheta;s) \cdot \dHs = \widehat{I}_2 (\hbtheta;s) \cdot \hdHs = 1.\]
\end{thm}
\smallskip

It is worth noting that this theorem is similar to Shimura's Theorem \ref{thm:globalratinality} for the case $r=2$.
At the conclusion of Section \ref{subsection:globalHei}, we establish that Theorem \ref{thm:global:Heisenbergex:intro} recovers Shimura's Theorem for $r=2$ via the endomorphism $\widehat{\phi}$ introduced in Definition \ref{df:extenstionmorphisms}.

The proof is essentially done by using some results of our previous study \cite{Hy} which is described in Section \ref{subsection:localHei}. There is no great difficulty in proving the claims stated in this study. Rather, it is important to note a natural generalization of series of \cite{H, S2, T}, and a concise identity given for a formal Dirichlet series whose coefficients are not always commutative (cf. Remark \ref{rem:noncommutativity}).

In \cite{Hy, Hy2} and this study, the case of the Heisenberg Lie algebra is considered as a first step. 
The author expects that many new Hecke rings will appear in the class of the Hecke rings of this study.
Further study of these Hecke rings is now in progress by the author. 
In \cite{Hy3}, the author investigated the Euler factor of $\hdls$ at each prime number in the case where $L$ is a higher Heisenberg Lie algebra.

Tamagawa \cite{T}, by using his Hecke theory, further investigated certain zeta functions, and proved that each of them is an entire function and has a functional equation. However, even in the Hecke rings associated with the Heisenberg Lie algebra, no analogue has been found. Further research is needed to find such applications to number theory.

It should be mentioned that our series $\dls$  and $\hdls$ are related to the zeta functions of groups and rings introduced by Grunewald, Segal and Smith \cite{GSS}. 
Let $G$ be a torsion-free finitely generated nilpotent group, and $\widehat{G}$ its profinite completion.
Denote by $\Si_n(G) \ (\text{resp.}\ \Sh_n(G))$ the family of subgroups $H$ of $G$ of index $n$ such that there is an isomorphism $H \cong G$ of groups $(\text{resp.}\ \widehat{H} \cong \widehat{G} \ \text{of topological groups})$.
The zeta functions $\zeta^i_G(s)$ and $\zeta^{\wedge}_G(s)$ of $G$ were defined in \cite{GSS} as follows:
\[\zeta^i_G(s) = \sum_{n > 0} \#\Si_n(G) n^{-s}, \quad \zeta^{\wedge}_G(s) = \sum_{n > 0} \# \Sh_n(G) n^{-s},\]
where $s$ is a complex variable. 

As an analogue of them, one can define the zeta functions of $L$.
Let $\hZ$ be the profinite completion of $\Z$, and set $\hl = L\otimes\hZ$. 
Denote by $\Si_n(L) \ (\text{resp.}\ \Sh_n(L))$ the family of subalgebras $M$ of $L$ of index $n$ such that there is an isomorphism $M \cong L$ as algebras $(\text{resp.}\ M\otimes \hZ \cong \hl \ \text{as algebras over $\hZ$})$.
We set 
$\aniL = \# \Si_n(L)$ and $\anhL = \# \Sh_n(L)$ for each $n$.
The zeta functions $\zetaiLs$ and $\zetahLs$ of $L$ are defined as follows:
\[\zetaiLs = \sum_{n > 0} \aniL n^{-s}, \quad \zetahLs = \sum_{n > 0} \anhL n^{-s}.\]
The zeta function $\zetahLs$ was also introduced in $\cite{GSS}$, 
and is called \textit{pro-isomorphic zeta function} $\zetah{L}(s)$ of $L$ in \cite{BGS} and \cite{BKO}. Although there are few papers on $\zetaiLs$, it is a natural analogue of $\zeta^i_G(s)$. We call $\zetaiLs$ the \textit{isomorphic zeta function} of $L$.

As we mention in Section \ref{section:zetafunc}, 
$\zetaiLs$ and $\zetahLs$ equal the coefficient-wise images of $\dls$ and $\hdls$ 
under the degree maps on $\rl$ and $\hrl$, respectively. For the definition of the degree map on a Hecke ring, see Section \ref{section:heckerings}.
Moreover, at the end of Section \ref{section:zetafunc}, we prove that
our Theorem \ref{thm:global:Heisenbergex:intro} derives the explicit formulae for $\zetai{\Hei}(s)$ and $\zetah{\Hei}(s)$ via the degree map on $\hrH$ as follows:
\[\zetaiHs = \zetahHs = \zeta(2s-2)\zeta(2s-3),\]
where $\zeta(s)$ is the Riemann zeta function. 

This identity is essentially due to Grunewald, Segal and Smith \cite[Theorem 7.6]{GSS}.
Precisely, for the free nilpotent group $\gFr = \gFr_{c,g}$ of class $c$ on $g$-generators, 
the identity $\zetai{\gFr}(s) = \zetah{\gFr}(s)$ and the explicit formulae for them were obtained. 
For the free nilpotent Lie algebra $\Fr = \Fr_{c,g}$ of class $c$ on $g$-generators,
by an argument essentially equivalent to that of \cite{GSS}, Berman, Glazer, and Schein proved in \cite[Theorem 5.1]{BGS} that
$\zetah{\Fr}(s)$ equals $\zetai{\gFr}(s)$ and $\zeta^{\wedge}_{\gFr}(s)$
(cf. Theorem \ref{thm:zeta:free:GSS-BGS}).
In Proposition \ref{prop:freenilp}, the equality $\zeta^{\wedge}_{\Fr}(s)  = \zeta^{i}_{\Fr}(s)$ is verified in a similar way as in the proof of \cite[Theorem 7.6]{GSS}. 
As a result, the equality $\zetai{\Fr}(s)  = \zetah{\Fr}(s) = \zetah{\gFr}(s) = \zetai{\gFr}(s)$ holds.

Write $\gHei$ for the Heisenberg group $\gFr_{2,2}$,  and focus on the identity $\zetaiHs = \zetahHs = \zeta^{\wedge}_{ \gHei }(s) = \zeta^i_{ \gHei }(s)$.
Another generalization of the identity $\zetahHs = \zeta^{\wedge}_{ \gHei }(s)$ is known.
Suppose that the nilpotent class of $G$ is $2$. Define the Lie algebra $L(G)$ as $(G/Z) \oplus Z$ 
with the usual Lie bracket operation induced by the commutator in $G$,
where $Z$ is the center of $G$. Then, we have $L(\gHei) = \Hei$, and the identity $\zeta^{\wedge}_{G}(s) = \zeta^{\wedge}_{L(G)}(s)$ is known to hold (cf. \cite[Section 1.1]{BGS} or \cite[Section 1.2.2]{Sau}). 
On the other hand, $\zeta^{i}_{G}(s) =\zeta_{G}^{\wedge}(s)$ does not hold in general. Indeed, Theorems 7.1 and 7.3 of \cite{GSS} provide a counter-example. 
The equality $\zeta^i_{G}(s) = \zeta^i_{L(G)}(s)$ is proved in Corollary \ref{cor:isomZeta}, and thus $\zetah{L(G)}(s) = \zeta^i_{L(G)}(s)$ does not always hold.

For pro-isomorphic zeta functions of Lie algebras, Berman, Glazer, and Schein \cite{BGS} further investigated.
The explicit formula for $\zetahLs$ was shown in \cite[Section 5]{BGS}, specifically for $L$ belonging to a certain class of Lie algebras over the integer rings of number fields. 
So far, we have not found any formulae for $\dls$ and $\hdls$ that recover their formulae except for this study.

The contents of this paper are organized as follows.
In Section \ref{section:heckerings}, we review the Hecke rings $\rl$ and $\hrl$. In Section \ref{section:seriesalg}, the formal series $\dls$ and $\hdls$ are introduced.
In Section \ref{section:Zr}, the case of $L = \Z^r$ is considered.
In Section \ref{section:Hei}, we study the series $\dHs$ and $\hdHs$, and prove our main theorem.
In Section \ref{section:zetafunc}, our series $\dls$ and $\hdls$ are related to the isomorphic zeta function $\zetaiLs$  and the pro-isomorphic zeta function $\zetahLs$ of $L$, respectively. Subsequently, we prove that
our main theorem also recovers the explicit formulae for $\zetaiHs$ and $\zetahHs$.
Finally, in Section \ref{section:isomZeta}, we observe the isomorphic zeta functions in the cases of the free nilpotent Lie algebras and class-2 nilpotent Lie algebras.

%%%%%%%%%%%%%%%%%%%%%%%%%%
%Section
%%%%%%%%%%%%%%%%%%%%%%%%%%

\section{Hecke rings associated with algebras}\label{section:heckerings}

First, we briefly recall the definition of Hecke rings and their degree maps.
For more details, refer to \cite[Chapter 3]{S2}.
Let $G$ be a group, $\mo$ be a submonoid of $G$, and $\gr$ be a subgroup of $\mo$. We assume that 
the pair $(\gr, \mo)$ is a double finite pair; that is, for all $A \in \mo$, 
$\gr \bs \gr A \gr$ and $\gr A \gr / \gr$ are finite sets.
Then, one can define the Hecke ring $R = R(\gr, \mo)$ associated with the pair $(\gr, \mo)$ as follows:
\begin{itemize}\setlength{\itemsep}{0mm}
\item The underlying abelian group is the free abelian group on the set $\gr \bs \mo /\gr$.
\item For all $A,B\in \mo$, the product of $\gr A \gr$ and $\gr B \gr$ is defined to be
\[\sum_{\gr C \gr \in \gr \bs \mo /\gr}
\#
\{
\gr \beta \in \gr \bs \gr B \gr \ |\ C\beta^{-1} \in \gr A\gr 
\}
\cdot \gr C \gr.
\]
\end{itemize}

For every $A \in \mo$, write $T_{\gr, \mo}(A)$ for the element $\gr A \gr$ of $R$.
We define the degree map on $R$ to be the additive map $\degr{R}: R \to \Z$ such that
$T_{\gr, \mo}(A)^{\degr{R}} = \# \gr \bs \gr A \gr$ for every $A \in \mo$.
Notably, it is known that $\degr{R}$ forms a ring homomorphism.

Let $p$ be a prime number, and let $L$ be as in Section \ref{section:intro}. We next recall the Hecke rings associated with $L$ introduced in \cite{Hy2}.
Fix a $\Z$-basis of $L$, and let $r$ be the rank of $L$.
Then, $\Aut_{\Q}^{alg}(L \otimes \Q)$, $\Aut_{\Q_p}^{alg}(L \otimes \Q_p)$, $\End_{\Z}^{alg}(L)$, and $ \End_{\Z_p}^{alg}(L \otimes \Z_p)$ are all identified with subsets of $M_r(\Q_p)$. 
In \cite[Section 2]{Hy2}, the following notation was introduced:
\[
	\begin{array}{llllll}
		\gl &=& \Aut_{\Q}^{alg}(L \otimes \Q), & \glp &=& \Aut_{\Q_p}^{alg}(L \otimes \Q_p), \\ \rule{0ex}{3ex}
		 \dl &=& \End_{\Z}^{alg}(L) \cap \gl, & \dlp &=& \End_{\Z_p}^{alg}(L \otimes \Z_p) \cap \glp, \\ \rule{0ex}{3ex}
		 \gaml  &=& \Aut_{\Z}^{alg}(L), & \gamlp  &=& \Aut_{\Z_p}^{alg}(L \otimes \Z_p).
	\end{array}
\]
The global Hecke rings $\rl$ and the local Hecke ring $\rlp$ are the Hecke rings with respect to $(\gaml, \dl)$ and $(\gamlp,\dlp)$, respectively.

The other global Hecke ring $\hrl$ was introduced in \cite[Section 3]{Hy2}. Define the group $\hgl$ to be the restricted direct product of $\glp$ relative to $\gamlp$ for all prime numbers $p$, that is, the set of elements $(\alpha_p)_p$
of $\prod_p \glp$ such that $\alpha_p \in \gamlp$ for almost all $p$. The monoid $\hdl$ and the group $\hgaml$ denote $ \hgl \cap \prod_p\dlp$ 
and $\prod_p \gamlp$, respectively. Then, we write $\hrl$ for the Hecke ring with respect to $(\hgaml, \hdl)$. 

Section 3 of \cite{Hy2} described relations among these Hecke rings. The local Hecke ring $\rlp$ is related to the global Hecke ring $\hrl$ as follows:
\begin{prop}[{\cite[Proposition 3.1]{Hy2}}]\label{prop:structureofhrl} The following assertions hold:
\begin{enumerate}\setlength{\itemsep}{0mm}
\item The local Hecke ring $\rlp$ is regarded as a subring of $\hrl$ by the map induced by the natural inclusion of $\dlp$ into $\hdl$.
\item For each prime number $q$ with $p \not = q$, the local Hecke rings $\rlp$ and $\rlq$ commute with each other in $\hrl$.
\item $\hrl$ is generated by the family of local Hecke rings $\{ \rlp \}_p$ as a ring.
\end{enumerate} 
\end{prop}
\smallskip

For simplicity, we set
\[
	\tlp = T_{\gamlp, \dlp}, \quad \tl = T_{\gaml, \dl}, \quad \htl = T_{\hgaml, \hdl}.
\]
Then, we have $\tlp(\alpha) = \htl(\alpha)$ in $\hrl$ for each $\alpha \in \dlp$.
Let us relate the global Hecke rings $\hrl$ and $\rl$.
The map $\etal$ denotes the diagonal embedding of $\dl$ into $\prod_p \dlp$.
Then, we define the additive map $\eta^*_L : \hrl \to \rl$ given by $\htl(\halpha) \mapsto \sum_{\beta} \tl(\beta)$, where $\beta$ runs through a complete system of representatives of $\gaml \bs \etal^{-1}(\hgaml \halpha \hgaml) / \gaml$. Let us denote by $\etadl: \gaml \bs \dl \to \hgaml \bs \hdl$ the map induced by $\etal$. Then, the two global Hecke rings are related as follows:

\begin{lemma}[{\cite[Lemma 3.2]{Hy2}}]\label{lemma:bijectivityandmultiplicative}
If the map $\etadl$ is bijective, then $\etaul$ is multiplicative and injective.
\end{lemma}
From now on, we regard $\hrl$ as a subring of $\rl$ if $\etadl$ is bijective.
\smallskip

In the rest of this section, we relate $\hrl$ to the automorphism group of $\hl$.

\newcommand{\bQ}{\mathbf{Q}}

\begin{prop}\label{prop:hgl:anotherdef}
Let $\Af$ be the ring of finite adeles over $\Q$, and set $\bQ = \prod_p \Q_p$. 
Then, the objects $\hgl$, $\hdl$, and $\hgaml$ satisfy the following identities as subsets of $M_r(\bQ)$:
\[
	\begin{array}{llllll}
		 \hgl = \Aut_{\Af}^{alg}(L \otimes \Af), &
		 \hdl = \End_{\hZ}^{alg}(\hl) \cap \hgl, &
		\hgaml  = \Aut_{\hZ}^{alg}(\hl).
	\end{array}
\]
\end{prop}
\begin{proof}
The second and third identities are straightforward consequences of the fact that $\hl$ equals $\prod_p (L\otimes \Z_p)$. Let us prove the first identity.
Denote by $GL'_r(\bQ)$ the restricted direct product of $GL_r(\Q_p)$ relative to $GL_r(\Z_p)$ for all prime numbers $p$. Then, it is easy to see that $GL'_r(\bQ)$ coincides with $GL_r(\Af)$. 
And the group $\hgl$, by definition, equals the intersection of   $GL'_r(\bQ)$ and $\prod_p\glp$. 
Thus, we have \[\hgl = GL_r(\Af) \cap \prod_p\glp.\]

Since $L\otimes\bQ$ is identified with $\prod_p(L\otimes\Q_p)$, the group $\prod_p\glp$ coincides with $\Aut_{\bQ}^{alg}(L \otimes \bQ)$. Hence, we have
\[GL_r(\Af)\cap \prod_p\glp = GL_r(\Af) \cap \Aut_{\bQ}^{alg}(L \otimes \bQ) =  \Aut_{\Af}^{alg}(L \otimes \Af) .\]
This implies the first identity.
\end{proof}

%%%%%%%%%%%%%%%%%%%%%%%%%%
%Section
%%%%%%%%%%%%%%%%%%%%%%%%%%

\section{Formal power series and formal Dirichlet series associated with algebras}\label{section:seriesalg}

\newcommand{\hAln}{\widehat{\mathcal{A}}_L(n)}
\newcommand{\Aln}{\mathcal{A}_L(n)}
\newcommand{\Alppk}{\mathcal{A}_{\lp}(p^k)}
\newcommand{\Alp}{\mathcal{A}_{\lp}}

Let $p$ and $L$ be as in the previous section. We set $\lp = L \otimes \Z_p$.
In this section, the formal series $\plp(X)$, $\dls$, and $\hdls$ are defined. Subsequently, their relationship is described.
For a positive integer $n$ and a nonnegative integer $k$, we introduce the following notation:
\begin{align*}
	\hAln &= \left\{ \halpha \in \hdl \ \left| \ [ \hl : \hl^{\halpha}] = n \right. \right\},
	\ 
	\Aln = \left\{ \alpha \in \dl \ \left| \ [ L : L^{\alpha}] = n \right. \right\},\\
	\Alppk &= \left\{ \alpha \in \dlp \ \left| \ [ \lp : \lp^{\alpha}] = p^k \right. \right\},
	\end{align*}
where $\lp^{\alpha}$ is the image of $\lp$ under the endomorphism $\alpha$. Additionally, $\hl^{\halpha}$ and $L^{\alpha}$ are defined in a similar manner. Note that each element of $\hdl$ is regarded as an element of $\End_{\hZ}^{alg}(\hl)$ by Proposition \ref{prop:hgl:anotherdef}.

Now, the formal power series $\plp(X)$ is introduced. We define
\[
	\tlpk = \sum_{\alpha} \tlp(\alpha),\\
\]
where $\alpha$ runs through a complete system of representatives of $\gamlp \bs \Alppk /\gamlp$. The formal power series $\plp(X)$ is defined as the generating function of the sequence $\{\tlpk\}_k$; that is,
\[\plp(X) = \sum_{k\geq 0} \tlpk X^k.\]

Next, the formal Dirichlet series $\hdls$ and $\dls$ are defined. We set
	\[\htln = \sum_{\halpha} \htl(\halpha), \quad 
	\tln = \sum_{\alpha} \tl(\alpha),\]
	where $\halpha$ $(\text{resp.}\ \alpha)$ runs through a complete system of representatives of 
	$\hgaml \bs \hAln /\hgaml$ $(\text{resp.}\ \gaml \bs \Aln/\gaml)$.	
	The formal Dirichlet series $\hdls$ and $\dls$ are the generating functions of the sequences of $\{\htln\}_n$ and $\{\tln\}_n$, respectively; that is,
	\[\hdls = \sum_{n > 0} \htln n^{-s}, \quad \dls = \sum_{n > 0} \tln n^{-s}.\]

Next, $\plp(X)$ is related to $\hdls$. For each element $\halpha$ of $\hdl$, let $\alpha_p$ denote its $\dlp$component. Then, $\htl(\halpha) = \prod_p \tlp(\alpha_p)$ is obtained, where $p$ runs over all prime numbers. 
Here, this infinite product is meaningful since
its terms commute with each other according to Proposition \ref{prop:structureofhrl}, and almost all of them are equal to $1$.
Consequently, the following theorem is proven:

\begin{thm}\label{thm:eulerprodhdls}
The sequence $\{\htln\}_n$ is multiplicative, and the Euler product formula for $\hdls$ holds; that is, 
\[\hdls = \prod_p \plp(p^{-s}),\]
where $p$ runs through all prime numbers.
\end{thm}
\begin{proof}
It is easy to see that $\tlp(p^k) = \htl(p^k)$ in $\hrl$.
Since $\hl$ is isomorphic to $\prod_p \lp$,
it follows that $[\hl:\hl^{\halpha}] = \prod_p [\lp:\lp^{\alpha_p}]$ for each $\halpha \in \hdl$.
Hence, we have 
\[
	\hAln = \prod_p \Alp(p^{v_p(n)}),
\]
where $v_p$ is the $p$-adic valuation. This proves the theorem.
\end{proof}
\medskip

Finally, $\hdls$ is related to $\dls$ using the additive map $\etaul:\hrl \to \rl$.
It is evident that $\etaul$ maps $\htln$ to $\tln$ for each positive integer $n$. Thus, the Euler product formula for $\dls$ is proven.

\begin{thm}\label{thm:eulerproddls}
If the map $\etadl$ is bijective, then the sequence $\{\tln\}_n$ is multiplicative, and the Euler product formula for $\dls$ holds; that is, 
\[\dls = \prod_p \plp(p^{-s}).\]
\end{thm}
\begin{proof}
By assumption, $\hrl$ is considered as a subring of $\rl$. 
Since $\tln = \htln$ and $\dls = \hdls$, Theorem \ref{thm:eulerprodhdls} implies the desired result.
\end{proof}

%%%%%%%%%%%%%%%%%%%%%%%%%%
%Section
%%%%%%%%%%%%%%%%%%%%%%%%%%
\section{Case of the free abelian Lie algebra $\Z^r$}\label{section:Zr}

Using the notations in Section \ref{section:seriesalg}, the theory of the Hecke ring with general linear groups as reported by Hecke \cite{H}, Shimura \cite{S2}, and Tamagawa \cite{T} is considered.

Let $r$ be a positive integer. Clearly, $\gZr$ and $\gZrp$ are identified with $GL_r(\Q)$ and $GL_r(\Q_p)$, respectively.
Similarly, we have $\dZr = M_r(\Z)\cap GL_r(\Q)$, 
$\dZrp = M_r(\Z_p)\cap GL_r(\Q_p)$, 
$\gamZr = GL_r(\Z)$, and,
$\gamZrp = GL_r(\Z_p)$.
Thus, the Hecke rings $\rZr$ and $\rZrp$ coincide with the Hecke rings treated in \cite{H}, \cite{S2}, and \cite{T}. Furthermore, the Hecke ring $\hrZr$ is identified with $\rZr$ as follows:
\begin{prop}\label{prop:isomrzr}
The map $\etauZr:\hrZr \to \rZr$ is an isomorphism.
\end{prop}
\begin{proof}
Lemma 3.3 of \cite{Hy2} implies that $\etauZr$ is an injective homomorphism.
Moreover, the map $\gamZr \bs \dZr / \gamZr \to \hgamZr \bs \hdZr / \hgamZr$
induced by $\etaZr$, is bijective according to the elementary divisor theorem.
Note that, in \cite{Hy2}, $\Z^r$ is defined as the ring of the direct sum of $r$-copies of $\Z$, which is incorrect. It is correct to define $\Z^r$ as the abelian free Lie algebra of rank $r$, as in the present study.
\end{proof}

Certainly, the formal power series $\pZrp$(X) equals the local Hecke series treated in \cite{H} and \cite{T}. The following theorem was proved:
\begin{thm}[{\cite[Satz 14]{H}, \cite[Theorem 3]{T}}]\label{thm:rationality}
Let \[\trpi = \gamZrp \diag[1,...,1,\overbrace{p,...,p}^i] \gamZrp\] for each $i$ with $1 \leq i \leq r$. Then, the following assertions hold:

	\begin{enumerate}\setlength{\itemsep}{0mm}
		\item $\rZrp$ is the polynomial ring over $\Z$ in variables $\trpi$ with $1 \leq i \leq r$.
		\item The series $\pZrp(X)$ is a rational function over $\rZr$, more precisely, 
		\[f_{r,p}(X)\pZrp(X) = 1,\]
		where $f_{r,p}(X) = \sum_{i=0}^{r}(-1)^ip^{i(i-1)/2}\trpi X^i$.
		Particularly,
		\[f_{2,p}(X) = 1 - T_{2,p}^{(1)}X + pT_{2,p}^{(2)}X^2. \]
	\end{enumerate}
\end{thm} 
%\medskip

\begin{rem}
Theorem \ref{thm:rationality} in the case $r=2$ was proved in \cite{H}.
For arbitrary $r$, it was demonstrated in \cite{T}.
\end{rem}
\medskip

The series $\dZrs$ is none other than the formal Dirichlet series treated in \cite[Chapter 3]{S2}. Since $\etadZr$ is bijective, the following theorem is obtained:
\begin{thm}\label{thm:eulerproductZr}
The Euler product formulae for $\dZrs$ and $\hdZrs$ hold; i.e., 
\[\dZrs = \hdZrs = \prod_p \pZrp(p^{-s}),\]
where $p$ runs through all prime numbers. 
\end{thm}
\begin{proof}
It is an immediate consequence of Theorems \ref{thm:eulerprodhdls} and \ref{thm:eulerproddls}.
\end{proof}
\medskip

Therefore,  the following theorem is obtained: 
\begin{thm}[{\cite[Theorem 3.21]{S2}}]\label{thm:globalratinality}
Define $I_r(s)$ to be the infinite product $\prod_pf_{r,p}(p^{-s})$. Then, the following is obtained:
\[I_r(s)\dZrs = I_r(s)\hdZrs = 1.\]
\end{thm}
\begin{proof}
This follows from Theorems \ref{thm:rationality} and \ref{thm:eulerproductZr}.
\end{proof}

%%%%%%%%%%%%%%%%%%%%%%%%%%
%Section
%%%%%%%%%%%%%%%%%%%%%%%%%%
\section{Case of the Heisenberg Lie algebra}\label{section:Hei}

This section studies the proposed series in the case of the Heisenberg Lie algebra $\Hei$.

%%%%%%%%%%%%%%%%%%%%%%%%%%
%Subsection
%%%%%%%%%%%%%%%%%%%%%%%%%%
\subsection{Local properties}\label{subsection:localHei}

Let us recall the main theorem of \cite{Hy}. 
For an element $A$ of $\gZtp$ and an element $\va$ of $\Q_p^2$, denote by $(A, \va)$ the element 
$
\begin{pmatrix}
	A & \va \nonumber \\
	0 \ 0 & |A|  \\ 
\end{pmatrix} 
$
of
$GL_3(\Q_p)$, where $|A|$ means the determinant of the matrix $A$.
Fix a system $\{x_1,x_2\}$ of free generators of $\Hei$. Then, the set $\{x_1,x_2,[x_1,x_2]\}$ forms a basis of $\Hei$. Hence, the group $\gHp$ is identified with the following subset of $GL_3(\Q_p)$:
\[
\left\{
(A,\va) \ 
\left|
\ 
A \in \gZtp, \ \va \in \Q_p^2
\right.
\right\}.
\]
In addition, an element $(A,\va)$ of $\gHp$ is contained in $\dHp$ (resp. $\gamHp$) if and only if
$A$ is in $\dZtp$ (resp. $\gamZtp$), and $\va$ is in $\Z_p^2$. 

The following three ring homomorphisms $s$, $\phi$, and $\theta$ were introduced in \cite[Section 6]{Hy}:
\begin{df}\label{df:local:maps}
For simplicity, we put $\deg = \degr{\rHp}$.
The ring homomorphisms $s: \rZtp \to \rHp$, $\phi: \rHp \to \rZtp$, and $\theta: \rHp \to \rHp$ are defined by
\begin{align*}
\tZtp(A)^s &= \tHp(A,\vzero) \ \text{for each $A \in \dZtp$},\\
\tHp(A,\va)^\phi &= \frac{\tHp(A,\va)^{\deg}}{\tHp(A,\vzero)^{\deg}}\tZtp(A) \ \text{for each $(A,\va) \in \dHp$},\\
\tHp(A,\va)^\theta&= \frac{\tHp(A,\va)^{\deg}}{\tHp(A,p\va)^{\deg}}\tHp(A,p\va) \ \text{for each $(A,\va) \in \dHp$}.\\
\end{align*}
\end{df}

\begin{rem}
Although the multiplicativity of $s$, $\phi$, and $\theta$ is not obvious by the definition, it was proved in \cite[Section 6]{Hy}.
\end{rem}
\smallskip

Some relations among the three ring homomorphisms are introduced.

\begin{prop}\label{prop:equalitesmorphisms}
The ring homomorphisms $s$, $\phi$, and $\theta$ satisfy the following properties:
\[\phi\circ s = id_{\rZtp},\quad \theta \circ s = s, \quad \phi \circ \theta = \phi.\]
\end{prop}
\begin{proof}
It is an easy consequence of Definition \ref{df:local:maps}.
\end{proof}
\medskip
 
Our previous work \cite{Hy} defined the element $T_2(p^k)$ of $\rHp$ for each nonnegative integer $k$ as follows:
\[
T_2(p^k) = \sum_{(A,\va)} \tHp(A,\va),
\]
where $(A, \va)$ runs through a complete system of representatives of $\gamHp \bs \dHp / \gamHp$ satisfying
$v_p(|A|) = k$.
The formal power series $D_{2,2}(X)$ was defined as the generating function of the sequence $\{T_2(p^k)\}_k$; that is,
\[
D_{2,2}(X) = \sum_{k \geq 0} T_2(p^k) X^k.
\]

The main theorem of our previous work \cite{Hy} is as follows:

\begin{thm}[{\cite[Theorem 7.8]{Hy}}]\label{thm:local:Heisenberg:previous}
Let $T_{2,p}^{(1)}$ and $T_{2,p}^{(2)}$ be as in Theorem \ref{thm:rationality}.
For simplicity, let us set $T_p(1,p) = T_{2,p}^{(1)}$ and $T_p(p,p) = T_{2,p}^{(2)}$. Define $Y=pX$. Then, $D_{2,2}(X)$ satisfies the following identity:
\[D_{2,2}(X)^{\theta^2} - T_p(1,p)^s D_{2,2}(X)^{\theta} Y + p T_p(p,p)^s D_{2,2}(X) Y^2=1,\]
where $D_{2,2}(X)^{\theta}$ is the coefficient-wise image of $D_{2,2}(X)$ under $\theta$, and $D_{2,2}(X)^{\theta^2}$ is defined similarly.
\end{thm}
\bigskip

The sequences $\{T_2(p^k)\}_{k\geq0}$ and $\{\tHp(p^k)\}_{k\geq0}$ are related as follows:
\begin{prop}\label{prop:relationT2tHpk}
$\tHp(p^{2k})=T_2(p^k)$ and $\tHp(p^{2k+1}) = 0$ for each $k$.
\end{prop}
\begin{proof}
It is evident that $v_p([\Hei_p : \Hei_p^{(A,\va)}]) = 2v_p(|A|)$ for every $(A,\va) \in \dHp$. This completes the proof.
\end{proof}
\medskip

The relation between $D_{2,2}(X)$ and $\pHp(X)$ is described as follows:
\begin{cor}  \label{cor:equality:dttphp}$D_{2,2}(X^2) = \pHp(X)$.
\end{cor}
\begin{proof}
It is an immediate consequence of the proposition above.
\end{proof}
\medskip

The Hecke ring $\rHp$ forms a ring over $\rZtp$ via the ring homomorphism $s$. 
Moreover, owing to the second identity of Proposition \ref{prop:equalitesmorphisms}, 
$\theta$ is a ring homomorphism over $\rZtp$.
Thus, $\rHp$ is a module (not a ring!) over the polynomial ring $\rZtp[\theta]$ in one variable $\theta$.
Further, the maps $s$, $\phi$, and $\theta$ depend on $p$. Subsequently, we set $s_p=s$, $\phi_p=\phi$, and $\theta_p = \theta$.
Therefore, Theorem \ref{thm:local:Heisenberg:previous} can be rewritten as follows:
\begin{thm}\label{thm:local:Heisenbergex}
Let $f_{2,p}(X)$ be as in Theorem \ref{thm:rationality}, and let us keep the notation of Theorem \ref{thm:local:Heisenberg:previous}.
Then, $\pHp(X)$ satisfies the following identity:
\[g_{2,p}(\theta_p;pX^2)\pHp(X)=1,\]
where 
\[g_{2,p}(\theta_p;X) = \theta_p^2 \cdot f_{2,p}(X/\theta_p) = \theta_p^2 - T_p(1,p) \theta_p X + p T_p(p,p)X^2.\]
\end{thm}
\begin{proof}
Clear.
\end{proof}
\medskip

We have just introduced the three ring homomorphism, of which $\phi_p$ has not been used so far. 
In fact, it has been shown that $\phi_p$ plays a role establishing the relationship between $\pHp(X)$ and $\pZtp(X)$ as follows:
\begin{thm}[{\cite[Theorem 7.5]{Hy}}]\label{thm:local:HeiZpt}
%\begin{align*}
$\pHp(X)^{\phi_p} = \pZtp(pX^2). $
%\end{align*}

\end{thm}

%%%%%%%%%%%%%%%%%%%%%%%%%%
%Subsection
%%%%%%%%%%%%%%%%%%%%%%%%%%
\subsection{Global properties}\label{subsection:globalHei}
In this subsection, the Dirichlet series $\dHs$ and $\hdHs$ are considered. 
Since the bijectivity of $\etadH$ was proved in \cite[Lemma 3.4]{Hy2}, the map $\etauH$ is an injective ring homomorphism.
Moreover, the nonsurjectivity of $\etauH$ was shown in \cite[Section 4]{Hy2}. Hence, the global Hecke ring $\hrH$ is a proper subring of $\rH$.
However, the following theorem can be obtained:

\begin{thm}\label{thm:eulerproducthei}
The Euler product formulae for $\dHs$ and $\hdHs$ hold; that is, 
\[\dHs = \hdHs = \prod_p \pHp(p^{-s}).\]
\end{thm}
\begin{proof}
It is an immediate consequence of Theorems \ref{thm:eulerprodhdls} and \ref{thm:eulerproddls}.
\end{proof}
\medskip

The ring homomorphisms $\hs$, $\hphi$, and $\htheta_p$ are defined as follows: 

\begin{df}\label{df:extenstionmorphisms}
The ring homomorphisms $\hs:\hrZt \to \hrH$ and $\hphi:\hrH \to \hrZt$ are defined by
\begin{align*}
	\htZt(\widehat{A})^{\hs} = \prod_p \tZtp(A_p)^{s_p} \text{ for each $\widehat{A} \in \hdzt$},\\
	\htH(\halpha)^{\hphi } = \prod_p \tHp(\alpha_p)^{\phi_p} \text{ for each $\halpha \in \hdH$},
\end{align*}
where $A_p$ $(\text{resp.} \ \alpha_p)$ is $\dZtp$ $(\text{resp.} \  \dHp)$ component of $\widehat{A}$ $(\text{resp.}\ \halpha)$ for each $p$.

The ring homomorphism $\htheta_p:\hrH \to\hrH$ is defined by
\[\htH(\halpha)^{\htheta_p} = \tHp(\alpha_p)^{\theta_p} \cdot \prod_{q\not = p} \tHq(\alpha_q) \ \text{ for each $\halpha \in \hdH$}.
\]
\end{df}
\medskip

Consequently, the following proposition is obtained: 
\begin{prop}\label{prop:equalitesglobalmorphisms}
The following equalities hold:
\begin{enumerate}\setlength{\itemsep}{0mm}
\item $\hphi \circ \hs = id_{\hrZt}$,
\item $\htheta_p \circ \hs = \hs$ and $\hphi \circ \htheta_p = \hphi$ for each $p$,
\item $\htheta_p \circ \htheta_q = \htheta_q \circ \htheta_p$ for any two prime numbers $p$, $q$.
\end{enumerate}
\end{prop}
\begin{proof}
It is an easy consequence of Proposition \ref{prop:equalitesmorphisms}.
\end{proof}
\medskip

From the proposition above, it is evident that the Hecke ring $\hrH$ is a ring over $\hrZt$ by $\hs$, and that $\htheta_p$ is a ring homomorphism over $\hrZt$ for each $p$.
Set $\hbtheta = (\htheta_p)_p$, and let $\hrZt[\hbtheta]$ be the polynomial ring over $\hrZt$ in infinitely many variables $\hbtheta$.
Then, the Hecke ring $\hrH$ is an $\hrZt[\hbtheta]$-module.

Now, the following theorem is proven, analogous to Theorem \ref{thm:globalratinality}:

\newcommand{\pHq}{P_{\Hei_q}}

\begin{thm}\label{thm:globalequality}
$\widehat{I}_2 (\hbtheta;s) $ is defined as the infinite product \[\prod_p g_{2,p}(\htheta_p;p^{1-2s}).\]
Then, the following is obtained:
\[\widehat{I}_2 (\hbtheta;s) \cdot \dHs = \widehat{I}_2 (\hbtheta;s) \cdot \hdHs = 1.\]
\end{thm}
\begin{proof}
Theorem \ref{thm:eulerproducthei} implies that
\[\hdHs = \prod_p\pHp(p^{-s}).\]
Let us fix a prime number $p$. Then, the following is obtained:
\[\hdHs^{\htheta_p} = \pHp(p^{-s})^{\theta_p} \cdot \prod_{q\not=p}\pHq(q^{-s}).\]
In addition, $\rHp$ and $R_{\mathcal{H}_q}$ commute with each other in $\hrH$ for any prime number $q$ different from $p$. Hence, for each element $\mathfrak{a}_p$ of $\rZtp$, we have
\[\mathfrak{a}_p\cdot (\hdHs) = \left(\mathfrak{a}_p\cdot \pHp(p^{-s})\right) \cdot \prod_{q\not=p}\pHq(q^{-s}).\]
Therefore,
\[\widehat{I}_2 (\hbtheta;s) \cdot \hdHs = \prod_{p}\left(g_{2,p}(\theta_p;p^{1-2s})\cdot \pHp(p^{-s})\right).\]
Subsequently, Theorem \ref{thm:local:Heisenbergex} implies that the right-hand side of the equality above is $1$, which completes the proof. 
\end{proof}
\medskip

In the remainder of this section, it is shown that Theorem \ref{thm:globalequality} recovers Shimura's Theorem \ref{thm:globalratinality}.
Let $\hpsi:\hrZt[\hbtheta] \to \hrZt$ be the ring homomorphism over $\hrZt$ satisfying $(\htheta_p)^{\hpsi} = 1$ for all $p$.
Then $\hphi$ and $\hpsi$ are compatible; that is, 
\[(\mathfrak{a} \cdot \mathfrak{A})^{\hphi} = \mathfrak{a}^{\hpsi} \cdot \mathfrak{A}^{\hphi} \quad \text{for any $\mathfrak{a} \in \hrZt[\hbtheta]$ and any $\mathfrak{A} \in \hrH$},\]
which follows from Proposition \ref{prop:equalitesglobalmorphisms}. Consequently, the following proposition is proven:
\begin{prop}\label{prop:exampleofphipsi}
The following identities hold:
\begin{enumerate}\setlength{\itemsep}{0mm}
\item $\widehat{I}_2 (\hbtheta;s)^{\hpsi} = I_2(2s-1)$,
\item $\hdHs^{\hphi} = \hdZt(2s-1)$.
\end{enumerate}
\end{prop}
\begin{proof}
It is evident that $g_{2,p}(\theta_p;pX^2)^{\hpsi}= g_{2,p}(1;pX^2) = f_{2,p}(pX^2)$, which implies the first equality.
The second one follows from Theorems \ref{thm:local:HeiZpt}, \ref{thm:eulerproducthei} and \ref{thm:eulerproductZr}.
\end{proof}
\medskip

Therefore, it is concluded that Theorem \ref{thm:globalequality} recovers Theorem \ref{thm:globalratinality} when $r=2$:
The map $\hphi$ is applied to the equality in Theorem \ref{thm:globalequality}. 
Then, Proposition \ref{prop:exampleofphipsi} and the compatibility of $\hphi$ and $\hpsi$ imply that
\[
I_2(2s-1) \hdZt(2s-1) = 1.
\]

\begin{rem}\label{rem:noncommutativity}
As shown in Theorem 7.3 of \cite{Hy}, the coefficients of the Hecke series $D_{2,2}(X)$ are not necessarily commutative. Thus, neither are those of $\pHp(X)$ and $\dHs$.
\end{rem}

%%%%%%%%%%%%%%%%%%%%%%%%%%
%Section
%%%%%%%%%%%%%%%%%%%%%%%%%%
\section{Zeta functions of algebras}\label{section:zetafunc}
Let us return to the case where $L$ is as in Section \ref{section:intro}.
In this section, our series $\dls$ and $\hdls$ are related to the isomorphic zeta function $\zetaiLs$  and pro-isomorphic zeta function $\zetahLs$ of $L$, respectively.

Denote by $\tSh_n(L)$ the family of $\hZ$-subalgebras $\mathcal{M}$ of $\hl$ of index $n$ such that there is an isomorphism $\mathcal{M} \cong \hl$ of algebras over $\hZ$.
Then, the maps $\gaml \bs \Aln \to \Si_n(L)$ given by $\gaml\alpha \mapsto L^{\alpha}$ and $\hgaml \bs \hAln \to \tSh_n(L)$ given by $\hgaml\halpha \mapsto \hl^{\halpha}$ are both bijective. 
Since $L$ is free of finite rank as an abelian group, 
one verifies in a similar way as in the proof of Proposition 1.2 of \cite{GSS} that the maps $\Sh_n(L) \to \tSh_n(L)$ 
defined by $M \mapsto M\otimes \hZ$ and $\tSh_n(L) \to \Sh_n(L)$ defined by  $\mathcal{M} \mapsto \mathcal{M} \cap L$ are inverse to each other, 
in particular, one has $\#\tSh_n(L) = \#\Sh_n(L) = \anhL$.
Hence, it follows that, for each $n$, 
\[\tln^{\degr{\rl}} = \aniL , \quad \htln^{\degr{\hrl}} = \anhL.\]
Thus, $\dls$ and $\hdls$ are related to $\zetaiLs$ and $\zetahLs$ as follows:
\[\dls^{\degr{\rl}} = \zetaiLs, \quad \hdls^{\degr{\hrl}} = \zetahLs.\]

By definition, we have $\degr{\hrl}|_{\rlp} = \degr{\rlp}$. Moreover, $\degr{\hrl}$ and $\degr{\rl}$ are related as follows:
\begin{prop}
If the map $\etadl$ is bijective, then we have
	$\degr{\hrl} = \degr{\rl} \circ \etaul$,
that is, 
	$\degr{\rl} |_{\hrl} = \degr{\hrl}$.
\end{prop}
\begin{proof}
Since $\etadl$ is bijective, so is 
\[
	\etadl |_{\gaml \bs \etal^{-1}(\hgaml \halpha \hgaml)} : \gaml \bs \etal^{-1}(\hgaml \halpha \hgaml) \to \hgaml \bs \hgaml 	\halpha \hgaml.
\]
Thus, we have
$
	\# \hgaml \bs \hgaml \halpha \hgaml = \# \gaml \bs \etal^{-1}(\hgaml \halpha \hgaml),
$
which completes the proof.
\end{proof}
\medskip

The above proposition implies the following corollary:
\begin{cor}\label{cor:equality:iso-proisoZeta}
	In order that $\zetahLs = \zetaiLs$, it is necessary and sufficient that $\etadl$ is bijective.
\end{cor}
\begin{proof}
Sufficiency is an easy consequence of the above proposition and the equality $\dls = \hdls$. 
Let us prove necessity. Since $\etadl$ is injective, so is $\etadl |_{\gaml \bs \Aln}$ for each $n$.
If $\zetahLs = \zetaiLs$, then we have $\#\gaml \bs \Aln = \#\hgaml \bs \hAln$ for each $n$, and thus, $\etadl$ is bijective, 
\end{proof}
\medskip

Next, the case $L = \Z^r$ is considered. We make use of the following identity shown by Tamagawa \cite{T}:
\begin{thm}[{\cite[Corollary]{T}}]\label{thm:localratinality:deg}
Let $f_{r,p}(X)$ be as in Theorem \ref{thm:rationality}. 
Then, the following identity holds:
\[
	f_{r,p}(X)^{\degr{\hrZr}} = \prod_{0 \leq k \leq r-1}(1 - p^{k}X).
\]
\end{thm}

This theorem derives the following identity:
\begin{cor}\label{cor:globalratinality:deg}
Let $I_r(s)$ be as in Theorem \ref{thm:globalratinality}.
Then, the following identity holds:
\[
	I_r(s)^{\degr{\rZr}} = I_r(s)^{\degr{\hrZr}} =\prod_{0 \leq k \leq r-1} \zeta(s-k)^{-1},
\]
where $\zeta(s)$ is the Riemann zeta function.
\end{cor}
\begin{proof}
This follows from the multiplicativity of $\degr{\hrZr}$ and the above theorem. 
\end{proof}
\medskip

Theorem \ref{thm:globalratinality} and Corollary \ref{cor:globalratinality:deg} recover the explicit formulae for $\zetah{\Z^r}(s)$ and $\zetai{\Z^r}(s)$ proved in \cite{GSS}.

\begin{cor}[{\cite[Proposition 1.1]{GSS}}]\label{cor:zeta:Zr}
The following identity holds:
\[\zetai{\Z^r}(s) = \zetah{\Z^r}(s) = \prod_{0 \leq k \leq r-1} \zeta(s-k).\]
\end{cor}
\begin{proof}
It is immediately verified by applying the maps $\degr{\rZr}$ and $\degr{\hrZr}$ to the identity of Theorem \ref{thm:globalratinality}.
\end{proof}
\medskip

Next, the case $L = \Hei$ is investigated.
	Let $\psi:\Z[\hbtheta] \to \Z$ be the ring homomorphism satisfying $(\htheta_p)^{\psi} = 1$ for all $p$.
	Then, the coefficient-wise images of $g_{2,p}(\htheta_p;X)$ and $\widehat{I}_2 (\hbtheta;s)$
	under the homomorphism $(\degr{\hrH} \circ \hs) \otimes \psi : \hrZt [\hbtheta] \to \Z$ are as follows:
	
\begin{prop}\label{prop:hei:deg}
	The following identities hold:
	\begin{align*} 
		 g_{2,p}(\htheta_p;X)^{(\degr{\hrH} \circ \hs) \otimes \psi} &= (1 - p^2X^2)(1 - p^3X^2), \\
		\widehat{I}_2 (\hbtheta;s)^{(\degr{\hrH} \circ \hs) \otimes \psi} &= \zeta(2s-2)^{-1}\zeta(2s-3)^{-1}.
	\end{align*}
\end{prop}
\begin{proof}
	By the identities (1)-(a) and (2)-(a) of \cite[Proposition 5.4]{Hy}, we have
	\[
		T_p(1,p)^{\degr{\hrH} \circ \hs} =p^2(1+p^{-1}), \quad T_p(p,p)^{\degr{\hrH} \circ \hs} = p^2,
	\] 
	which implies the first identity. The second is easily derived by the first one.
\end{proof}
%\medskip

\begin{rem}
	$\degr{\hrH} \circ \hs$ and $\degr{\hrZt}$ are slightly different: 
	For each $A \in \widehat{\Delta}_{\Z^2}$, it follows from \cite[Proposition 5.4]{Hy} and \cite[Theorem 3.24]{S2} that
	\[\htZt(\widehat{A})^{\degr{\hrH} \circ \hs} = [\hZ^2:(\hZ^2)^A]\cdot\htZt(\widehat{A})^{\degr{\hrZt}}.\] 
\end{rem}
\smallskip

Since $\degr{\hrH} \circ \htheta_p = \degr{\hrH}$, the ring homomorphism 
	$(\degr{\hrH} \circ \hs\ ) \otimes \psi$
	and $\degr{\hrH}$ are compatible.
Therefore, the explicit formulae for $\zetaiHs$ and $\zetahHs$ shown in \cite{GSS} and $\cite{BGS}$ are recovered.
\begin{cor}[{\cite[Theorem 7.6]{GSS}}, {\cite[Theorem 5.1]{BGS}}]
	The following identity holds:
	\[\zetaiHs = \zetahHs = \zeta(2s-2)\zeta(2s-3).\]
\end{cor}
\begin{proof}
	It is an easy consequence of Proposition \ref{prop:hei:deg} and Theorem \ref{thm:globalequality}.
\end{proof}
%

%%%%%%%%%%%%%%%%%%%%%%%%%%%%%%%%%%%%%%
%section
%%%%%%%%%%%%%%%%%%%%%%%%%%%%%%%%%%%%%%
\section{Isomorphic zeta functions}\label{section:isomZeta}
In this section, we observe the isomorphic zeta functions in the cases of the free nilpotent Lie algebras and class-2 nilpotent Lie algebras.
%

%%%%%%%%%%%%%%%%%%%%%%%%%%%%%%%%%%%%%%
%
%%%%%%%%%%%%%%%%%%%%%%%%%%%%%%%%%%%%%%
\subsection{Case of the free nilpotent Lie algebras}\label{subsection:freenilp}
\newcommand{\ab}{\mathrm{ab}}
Let $\gFr = \gFr_{c,g}$ (resp. $\Fr=\Fr_{c,g}$) be the free nilpotent group (resp. Lie algebra) of class $c$ on $g$-generators. For a nilpotent Lie algebra $L$, denote by $\gamma_i(L)$ the $i$-th term of its lower central series,
and set $L^{\it ab} = L/[L,L]$. 
In this section, we describe the explicit formulae for the zeta functions of $\gFr$ and $\Fr$.
As mentioned in the introduction, they are essentially due to \cite{GSS}.
However, other literature does not deal with isomorphic zeta functions of algebras, and it is necessary to give a detailed proof of the explicit formula for $\zetai{\Fr}(s)$.

In the case $c=1$, we have $\gFr_{c,g} = \Fr_{c,g} = \Z^g$. Hence, the explicit formulae were obtained in Corollary \ref{cor:zeta:Zr}.
In the following, suppose that $c \geq 2$. 
The explicit formula for $\zetah{\Fr}(s)$ was proved in $\cite{BGS}$, and those for $\zetai{\gFr}(s)$ and $\zetah{\gFr}(s)$ were established in \cite{GSS}:
\begin{thm}[{\cite[Theorem 7.6]{GSS}, \cite[Theorem 5.1]{BGS}}]\label{thm:zeta:free:GSS-BGS}
Let $m_i$ be the rank of $\gamma_i(\Fr)/\gamma_{i+1}(\Fr)$ for each $ 1 \leq i \leq c$.
Define $\alpha = \frac{1}{g}\sum_{i=1}^c i m_i$, and $\beta = \sum_{i=2}^c i m_i$. Then, we have
\[\zetai{\gFr}(s) = \zetah{\gFr}(s) =  \zetah{\Fr}(s)  = \prod_{j=0}^{g-1}\zeta(\alpha s - \beta -j).\]
\end{thm}
\medskip
\begin{rem}
The following formula for $m_i$ was established in \cite[Satz 3]{Wit}:
\[m_i = \frac{1}{i} \sum_{j | i} \mu(j)g^{i/j},\]
where $\mu$ is the M\"obius function.
\end{rem}
\medskip

Therefore,  it remains to consider $\zetai{\Fr}(s)$. We prepare the following lemma which is a Lie algebra analogue of 
\cite[Theorem 1.8]{War}. The proof imitates the one of \cite[Chapter 1, Exercise 7]{Seg}:
\begin{lemma}\label{lemma:subalgeqalg}
Let $L$ be a nilpotent Lie algebra, and let $M$ be a subalgebra of $L$. If $L = M + [L,L]$, then $L = M$.
\end{lemma}
\begin{proof}
For each $i \geq 0$, denote by $Z_i$ the $i$-th upper central series of $L$.
Assume that $M$ was a proper subalgebra of $L$. 
Then, there exists $i > 0$ such that $Z_i + M = L$ and $Z_{i-1} + M \subsetneq L$.
Since $[L,L] = [Z_i + M,Z_i + M] \subset Z_{i-1} + M$, we have $ L = M + [L,L] \subset Z_{i-1}+M$, which is a contradiction.
\end{proof}
\medskip

Now, we prove the explicit formula for $\zetai{\Fr}(s)$.
\begin{prop}\label{prop:freenilp}
Let us keep the notation of Theorem \ref{thm:zeta:free:GSS-BGS}. Then, the following identity holds:
\[\zetai{\Fr}(s) = \zetah{\Fr}(s) = \prod_{j=0}^{g-1}\zeta(\alpha s - \beta -j).\]
\end{prop}
\begin{proof}
Theorem \ref{thm:zeta:free:GSS-BGS} reduces us to proving that $\zetai{\Fr}(s) = \zetah{\Fr}(s)$.
Let $M$ be an element of $\Sh_n(\Fr)$ with a positive integer $n$. It is sufficient to show that $M$ and $\Fr$ are isomorphic as Lie algebras.
Clearly, $\Fr^\ab\otimes \hZ$ and $M^\ab\otimes \hZ$ are isomorphic to $(\widehat{\Fr})^\ab$ and $(\widehat{M})^\ab$ as $\hZ$-modules, respectively.
By assumption, $(\widehat{\Fr})^\ab$ and $(\widehat{M})^\ab$ are isomorphic as $\hZ$-modules. 
Hence, $M^\ab\otimes \hZ$ is isomorphic to $\Fr^\ab\otimes \hZ = \hZ^g$ as $\hZ$-modules.
Since $M^\ab$ is finitely generated abelian group, it follows from the fundamental theorem of finitely generated abelian groups that $M^\ab$ is isomorphic to $\Z^g$.
Hence, we can take elements $x_1,...,x_g$ of $M$ such that their images under the canonical projection $M \to M^\ab$ form a basis of $M^\ab$, and it follows from Lemma \ref{lemma:subalgeqalg} that $x_1,...,x_g$ generate $M$.
By the universal property of $\Fr$, there exists a surjective homomorphism $\varphi: \Fr \to M$, and the induced map $\varphi \otimes id_{\hZ}: \widehat{\Fr} \to \widehat{M}$ is also surjective.
Since $\widehat{\Fr}$ and $\widehat{M}$ are isomorphic as algebras over $\hZ$,
the ranks of them over $\hZ$ are the same. Therefore, $\varphi \otimes id_{\hZ}$ is an isomorphism. 
The faithful flatness of $\hZ$ over $\Z$ implies the bijectivity of $\varphi$, which completes the proof.
\end{proof}

\newcommand{\SnG}{\mathcal{S}_n(G)}
\newcommand{\SnL}{\mathcal{S}_n(L)}

\subsection{Case of class-2 nilpotent Lie algebras}\label{subsection:classTwo}
Suppose that $L$ is a nilpotent Lie algebra of class $2$. 
By the class-two Lie correspondence of \cite[Section 3.1]{BKO}, there exists a unique torsion-free finitely generated nilpotent group $G$ of class $2$ up to isomorphism such that 
$L$ is isomorphic to $L(G) = (G/Z) \oplus Z$ as Lie algebras, where $Z$ is the center of $G$. Hence, we may identify $L$ with $L(G)$, and $Z$ is regarded as the center of $L$ by the map $z \in Z \mapsto (1,z) \in L$.
In this subsection, the equality $\zeta_G^i(s) = \zeta_L^i(s)$ is verified. 
To show this, Proposition \ref{prop:correspondence} below is essential. 
The proposition is mentioned in many references without proof, for example, \cite[Section 4]{GSS},\cite[Section 1.2.2]{Sau}, and \cite[Section 2.1]{BGS}. 
Although a proof is proposed in \cite[Proposition 3.1]{BKO}, it is incorrect (cf. Remark \ref{rem:BKO}). Therefore, it would be worthwhile to give a precise proof in this study.
\begin{prop}\label{prop:correspondence}
Let $n$ be a positive integer. Denote by $\SnG$ (resp. $\SnL$) the set of subgroups (resp.  subalgebras) of $G$ (resp.  $L$) of index $n$. Then, there exists a bijection $f_n: \SnG \to \SnL$ such that, for each $H \in \SnG$, its image is isomorphic to $L(H)$ as Lie algebras. In particular, we have $\#\SnG = \#\SnL$.
\end{prop}
Before proving the proposition, we deduce the following corollary which is our purpose of this subsection:
\begin{cor}\label{cor:isomZeta}
The equalities $\zetai{G}(s) = \zetai{L}(s)$ and $\zetah{G}(s) = \zetah{L}(s)$ hold.
\end{cor}
\begin{proof}
Let $H \in \SnG $, and put $M = f_n(H)$. Then, $M \cong L(H)$.
By the class-two Lie correspondence of \cite[Section 3.1]{BKO}, we see that $H \in \Si_n(G)$ if and only if $M \in \Si_n(L)$. Thus, we have $\#\Si_n(G) = \#\Si_n(L)$.
Since $L(\widehat{G}) = L \otimes \hZ$ and $L(\widehat{H}) = L(H) \otimes \hZ$, it follows from the local class-two Lie correspondence of \cite[Section 3.1]{BKO} that $H \in \Sh_n(G)$ if and only if $M \in \Sh_n(L)$. Thus, we have $\#\Sh_n(G) = \#\Sh_n(L)$, which completes the proof.
\end{proof}
\smallskip

\begin{rem}
It is mentioned without proof  in \cite[Section 1.1]{BGS} and \cite[Section 1.2.2]{Sau}
that the equality $\zetah{G}(s) = \zetah{L}(s)$ holds.
\end{rem}
\medskip

\newcommand{\SGAB}{\mathcal{S}_G(A,B)}
\newcommand{\SLAB}{\mathcal{S}_L(A,B)}

\newcommand{\tBG}{\tilde{B}_G}
\newcommand{\tBL}{\tilde{B}_L}
\newcommand{\bpiG}{\bar{\pi}_G}
\newcommand{\bpiL}{\bar{\pi}_L}

\newcommand{\bSGAB}{\overline{{\mathcal{S}}_G(A,B)}}
\newcommand{\bSLAB}{\overline{{\mathcal{S}}_L(A,B)}}

\newcommand{\SG}{\SGAB}
\newcommand{\SL}{\SLAB}

\newcommand{\pL}{\Pi_L}
\newcommand{\pG}{\Pi_G}

\newcommand{\isom}{\lambda}
\newcommand{\pisom}{\Lambda}

\newcommand{\tb}{\tilde{b}}

\newcommand{\Kill}{\varphi}

In order to prove Proposition \ref{prop:correspondence}, we introduce some notation. Let $\pi_G$ (resp. $\pi_L$) be the canonical projection $G \to G/Z$ (resp. $L \to G/Z$). 
Denote by $\Kill$ the map
$G/Z \times G/Z \to Z$ given by $(xZ, yZ) \mapsto [x,y]=x^{-1}y^{-1}xy$.
Since $G$ is of nilpotent class $2$, its derived group $[G,G]$ is contained in $Z$. 
Hence, $G/Z$ is abelian group, and we have $[xy,z] = [x,z][y,z], [x,yz] = [x,y][x,z]$ for any $x,y,z \in G$.
This implies that $\Kill$ is a $\Z$-bilinear form.

Further, let $A$ and $B$ be finite-index subgroups of $Z$ and $G/Z$, respectively. 
Denote by $\SGAB$ the set of subgroups $H$ of $G$ such that $H \cap Z = A$ and $\pi_G(H) = B$. 
Similarly, denote by $\SLAB$ the set of subalgebras $M$ of $L$ such that $M \cap Z = A$ and $\pi_L(M) = B$.
For any $H \in \SGAB$ and $M \in \SLAB$, we have $[G:H] = [L:M] = [Z:A][G/Z:B]$.
Hence, to prove Proposition \ref{prop:correspondence}, it is sufficient to show the following lemma:
\begin{lemma}\label{lemma:bij-SG-SL}
There exists a bijection $f_{A,B}: \SGAB \to \SLAB$ such that, for each $H \in \SGAB$, its image is isomorphic to $L(H)$ as Lie algebras.
\end{lemma}

To prove this lemma, we need to study $\SGAB$ and $\SLAB$.
First, the following three lemmas are shown. 
Although they hold in general for arbitrary torsion-free nilpotent groups (cf. Remark \ref{rem:arbitrarycase}), we provide a direct proof here for self-containedness:
\begin{lemma}\label{lemma:torsionfree:bilinear}
Let $x,y \in G$, and let $n$ be a positive integer. If $[x,y^n] =1$, then $[x,y]=1$.
\end{lemma}
\begin{proof}
Since $\varphi$ is bilinear, we have $1 = [x,y^n] = [x,y]^n$.
Since $G$ is torsion-free, we have $[x,y]=1$.
\end{proof}

\begin{lemma}\label{lemma:G-Z:torsionfree}
$G/Z$ is torsion-free.
\end{lemma}
\begin{proof}
Let $x,y \in G$, and suppose that $y^n \in Z$ for some $n>0$. 
By Lemma \ref{lemma:torsionfree:bilinear}, we have $[x,y]=1$, which derives that $y \in Z$. Thus, $G/Z$ is torsion-free.
\end{proof}

\begin{lemma}\label{lemma:ZH-A}
For each finite-index subgroup $H$ of $G$, the center $Z_H$ of $H$ equals $Z \cap H$.
\end{lemma}
\begin{proof}
It is immediately verified that $Z \cap H \subset Z_H$. Let $x$ be an element of $Z_H$, and let $y$ be an element of $G$. Since $H$ is of finite index, there exists a positive integer $n$ such that $y^n \in H$. 
Hence, we have $[x, y^n]= 1$, which implies $[x,y]=1$ by Lemma \ref{lemma:torsionfree:bilinear}. 
Thus, we have $Z_H \subset Z \cap H$.
\end{proof}
\medskip

\begin{rem}\label{rem:arbitrarycase}
Lemmas \ref{lemma:torsionfree:bilinear}, \ref{lemma:G-Z:torsionfree}, and \ref{lemma:ZH-A} hold for an arbitrary torsion-free nilpotent group $G'$.
Let $x,y \in G'$, and let $n$ be a positive integer. 
Suppose that $[x,y^n] =1$. Then, we have $(x^{-1} y x)^n = y^n$.
According to a result of Chernikov (cf. \cite[Theorem 4.10]{War}), the map $w \in G' \mapsto w^n \in G'$ is injective, which implies that $[x,y]=1$.
Thus, Lemma \ref{lemma:torsionfree:bilinear} holds for $G'$.
By using this, Lemmas \ref{lemma:G-Z:torsionfree} and \ref{lemma:ZH-A} for $G'$ are verified.
\end{rem}
\medskip

\newcommand{\settb}{\mathfrak{b} }
\newcommand{\Hb}{ H^\settb}
\newcommand{\Mb}{ M^\settb }

By Lemma \ref{lemma:G-Z:torsionfree}, $G/Z$ is a free abelian group of finite rank.
Let $d$ denote this rank. We next relate $(Z/A)^d$ to $\SGAB$ and $\SLAB$. 
Fix a subset $\settb = \{\tb_i\}_{i=1}^d$ of $G$ such that its image under $\pi_G$ forms a basis of $B$.  Set $b_i = \pi_G(\tb_i)$ for each $i$.
For an element $\Xi = (\xi_i)_i$ of $(Z/A)^d$, take an element $(z_i)_i$ of $Z^d$ satisfying $\xi_i = z_iA$ for each $i$.
Further, define $\Hb(\Xi)$ (resp. $\Mb(\Xi)$) to be the subgroup (resp. subalgebra) of $G$ (resp. $L$) generated by $A$ and $\{\tb_iz_i\}_i$ (resp. $\{(b_i,z_i)\}_i$). Since $A \subset Z$, the group $\Hb(\Xi)$ and the Lie algebra $\Mb(\Xi)$ are independent of the choice of $(z_i)_i$. 

If $\Kill(B,B) \subset A$, then, $[\tb_i,\tb_j]$ and $[(b_i,z_i),(b_j,z_j)]$ are contained in $A$ for any $i$, $j$.
Since $B$ is a free abelian group, each element of $\Hb(\Xi)$ (resp. $\Mb(\Xi)$) can be written uniquely as
a product (resp. sum)
$a\cdot \prod_i(\tb_i z_i)^{n_i}$ (resp. $(1,a) + \sum_i n_i(b_i,z_i)$),
where $(n_i)_i \in \Z^d$ and $a \in A$.
Hence, $\Hb(\Xi) \in \SGAB$, and $\Mb(\Xi) \in \SLAB$.
Moreover, $B \oplus A$ is a subalgebra of $L$, and a unique isomorphism $B \oplus A \to \Mb(\Xi)$ of Lie algebras
is determined by the inclusion $A \subset \Mb(\Xi)$ and the additive map $B \to \Mb(\Xi)$ defined by $b_i \mapsto (b_i,z_i)$ for each $i$.
\begin{lemma}\label{lemma:H-M-Xi}
The following assertions hold:
\begin{enumerate}
\item The following three conditions are equivalent:
\begin{itemize}
\item $\SGAB \not = \emptyset$,
\item $\SLAB \not = \emptyset$,
\item $\Kill(B,B) \subset A$.
\end{itemize}
\item
Suppose that $\Kill(B,B) \subset A$, and let $\Xi$ be an element  of  $(Z/A)^d$. Then, $\Hb(\Xi) \in \SGAB$, and $\Mb(\Xi) \in \SLAB$. Moreover, $L( \Hb(\Xi) )$ and $\Mb(\Xi)$ are isomorphic to the subalgebra $B \oplus A$ of $L$ as Lie algebras.
\item If $\Kill(B,B) \subset A$, then the maps $\lambda_G^{\settb}:(Z/A)^d \to \SG$ given by $\Xi \mapsto \Hb(\Xi)$, and $\lambda_L^{\settb}:(Z/A)^d \to \SL$ given by $\Xi \mapsto \Mb(\Xi)$ are bijective. In particular, we have $\#\SG = \#\SL = [Z:A]^d$.
\end{enumerate}
\end{lemma}
\newcommand{\vone}{\mathbf{1}}

\begin{proof}
\ 

1. 
If there exists $H \in \SG$, then $H/A$ is isomorphic to the abelian group $B$ by the canonical projection $G/A \to G/Z$. Hence, $\Kill(B,B) = [H,H] \subset A$.
Similarly, we have $\Kill(B,B) \subset A$ if $\SL \not = \emptyset$.
Conversely, if $\Kill(B,B)  \subset A$, then we have $\Hb(\vone) \in \SG$, and $\Mb(\vone) \in \SL$, where $\vone$ is the identity element of $(Z/A)^d$.

2. It is sufficient to prove that $L(\Hb(\Xi)) = B \oplus A$, which is an easy consequence of Lemma \ref{lemma:ZH-A}.

3. Let $H \in \SG$. Then, $H/A$ is isomorphic to $B$ by the canonical projection $G/A \to G/Z$.
Hence, there exists a unique element $\Xi =(\xi_i)_i$ of $(Z/A)^d$ 
such that $\{(\tb_iA)\cdot \xi_i\}_i$ forms a basis of $H/A$, and we have $H = \Hb(\Xi)$. Therefore, the map $\lambda_G^{\settb}$ is bijective.
In a similar way, for each $M \in \SL$,  there exists a unique element $\Xi =(\xi_i)_i$ of $(Z/A)^d$ 
such that the subset$\{(b_i,\xi_i)\}_i$ of $(G/Z) \oplus (Z/A) = L/A$ forms a basis of $M/A$, and we have $M = \Mb(\Xi)$.
Thus,  $\lambda_L^{\settb}$ is also bijective.
\end{proof}
\medskip

Now, a proof of Lemma \ref{lemma:bij-SG-SL} is obtained:
\begin{proof}[Proof of Lemma \ref{lemma:bij-SG-SL}]
It follows from Lemma \ref{lemma:H-M-Xi} that $\SGAB = \SLAB = \emptyset$ if $\Kill(B,B) \not \subset A$.
If $\Kill(B,B) \subset A$, then $f_{A,B} = \lambda_L^{\settb}\circ (\lambda_G^{\settb})^{-1}$ has the desired property by Assertions 2 and 3 of Lemma \ref{lemma:H-M-Xi}.
\end{proof}

\newcommand{\sG}{\scalebox{0.5}{$G$}}
\newcommand{\sL}{\scalebox{0.5}{$L$}}

\newcommand{\xsG}{x_{\sG}}
\newcommand{\xsL}{x_{\sL}}
\newcommand{\ysG}{y_{\sG}}
\newcommand{\ysL}{y_{\sL}}
\newcommand{\zsG}{z_{\sG}}
\newcommand{\zsL}{z_{\sL}}

\newcommand{\StwoG}{\mathcal{S}_4(G)}
\newcommand{\StwoL}{\mathcal{S}_4(L)}

In the rest of this subsection, we give an example of 
Lemmas \ref{lemma:bij-SG-SL} and \ref{lemma:H-M-Xi}. Subsequently, using this, we remark on Proposition 3.1 of \cite{BKO}.
\begin{ex}
Consider the case where $G$ is the Heisenberg group $\gHei$ with a free generating set $\{\xsG,\ysG\}$. 
Set $\settb = \{\xsG\ysG, \ysG^2\}$, 
$B = \langle \pi_G(\xsG\ysG)$, 
$\pi_G(\ysG^2) \rangle$, and
$A = \langle \zsG^2 \rangle$,
 where $z_{\sG}  = [x_{\sG},y_{\sG}] = x_{\sG}^{-1}y_{\sG}^{-1}x_{\sG}y_{\sG}$.
Then, we have $[G/Z : B] = [Z:A] = 2$, and hence, $\SG \subset \StwoG$, $\SL \subset \StwoL$.
Since $\Kill(B,B) = A$, it follows from Lemma \ref{lemma:H-M-Xi} that $\#\SG = \#\SL = 4$.

Now, $L=L(G)$ is the Heisenberg Lie algebra $\Hei$, and generated by $\xsL = (\pi_G(\xsG), 1)$ and $\ysL = (\pi_G(\ysG), 1)$. Put $\zsL = (1, \zsG)$. Then, $[\xsL,\ysL] = \zsL$.
By Lemma \ref{lemma:H-M-Xi}, 
$f_{A,B} = \lambda_L^{\settb}\circ (\lambda_G^{\settb})^{-1}$ is a bijection $\SG \to \SL$, and
we have $f_{A,B}(\Hb(\vone)) = \Mb(\vone)$, where $\vone$ is the identity element of $(Z/A)^2$. 
Moreover,
\begin{align*}
\Mb(\vone)&= \Z(\xsL + \ysL) + 2\Z\ysL + 2\Z\zsL \\
		& = \{ k\xsL + l \ysL + m\zsL \ |\ (k,l) \in \Z(1,1) + \Z(0,2),\ m \in 2\Z \},\\
\Hb(\vone) &= \{(\xsG\ysG)^k \ysG^{2l} \zsG^{2m} \ | \ k,l,m \in \Z\}. 
\end{align*}
\end{ex}

\begin{rem}\label{rem:BKO}
We keep the notation of the above example.
Proposition 3.1 of \cite{BKO} claims that, for each $n$, a one-to-one correspondence between $\SnG$ and $\SnL$ is induced by the bijection $f':G \to L$ defined by 
\[
f'(\xsG^k \ysG^l \zsG^m) =  k x_{\sL} + l y_{\sL} + mz_{\sL}
\quad 
\text{for $k,l,m \in \Z$}.
\]
However, it is not true. Indeed, $\Mb(\vone)$ corresponds to $H' = f'^{-1}(\Mb(\vone)) = \{ \xsG^k \ysG^l \zsG^m \ |\ (k,l) \in \Z(1,1) + \Z(0,2),\ m \in 2\Z \}$, however, $H'$ is not a group because $\xsG\ysG\in H'$, and $(\xsG\ysG)^2 = \xsG^2\ysG^2 \zsG^{-1} \not \in H'$.%
\end{rem}

%%%%%%%%%%%%%%%%%%%%%%%
%Acknowledgements
%%%%%%%%%%%%%%%%%%%%%%5 
\section*{Acknowledgements}
The author extends his sincere gratitude to the reviewers for their meticulous evaluation and numerous valuable suggestions. In particular, he is grateful to them for providing references \cite{BGS, BKO, GSS}, which are highly relevant to this study, and for their insightful suggestions for future research. 
In addition, he wishes to thank them 
for providing a direct proof of Lemma \ref{lemma:torsionfree:bilinear}
that does not rely on Chernikov's result.
He also would like to thank the editors for their helpful suggestions and advice.

%author, title, journal, number, number, year, page
\newcommand{\outbibitem}[7]
{
#1: #2, #3 {\bf #4}%
\ifthenelse{\equal{#5}{}}{}{(#5)}
(#6), #7. 
}
%author, title, publisher, pref, year
\newcommand{\outbibitembook}[5]
{
#1: #2, #3, #4, #5.
}

\textsc{Department of Health Informatics, Faculty of Health and Welfare Services Administration, Kawasaki University of Medical Welfare, Kurashiki, 701-0193, Japan}

{\it Email address}: \texttt{fumitake.hyodo@mw.kawasaki-m.ac.jp}

\end{document}